\documentclass[12pt]{amsart}
\usepackage{amsmath,amsthm,amsfonts,mathrsfs,amssymb,cite}

\newtheorem{thm}{Theorem}
\newtheorem{theorem}[thm]{Theorem}
\newtheorem{corollary}[thm]{Corollary}

\newtheorem{lemma}[thm]{Lemma}

\theoremstyle{definition}

\theoremstyle{remark}
\newtheorem{remark}[thm]{Remark}

\newcommand\ov{\overline}
\newcommand\dom{\operatorname {dom}}

\newcommand{\bN}{\mathbb N}
\newcommand{\bR}{\mathbb R}
\newcommand{\bZ}{\mathbb Z}

\newcommand{\sD}{\mathscr D}

\newcommand{\si}{\sigma}

\begin{document}
\title[Schr\"odinger operators with singular potentials]
{Self-adjointness of Schr\"odinger operators with singular potentials
}%
\author[R.~O.~Hryniv \and Ya.~V.~Mykytyuk]{Rostyslav O.~Hryniv \and Yaroslav V.~Mykytyuk}%
\address[R.H.]{Institute for Applied Problems of Mechanics and Mathematics,
3b~Naukova st., 79601 Lviv, Ukraine}%
\email{rhryniv@iapmm.lviv.ua}

\address[Ya.M.]{Lviv National University, 1 Universytetska st., 79602 Lviv, Ukraine}
\email{yamykytyuk@yahoo.com}

\thanks{}%

\date{March 21, 2012}%
\dedicatory{Dedicated to the memory of A.~G.~Kostyuchenko}%

\begin{abstract}
We study one-dimensional Schr\"odinger operators~$S$ with real-valued distributional potentials $q$ in $W^{-1}_{2,\mathrm{loc}}(\bR)$ and prove an extension of the Povzner--Wienholtz theorem on self-adjointness of bounded below~$S$ thus providing additional information on its domain. The results are further specified for $q\in W^{-1}_{2,\mathrm{unif}}(\bR)$.
\end{abstract}

\maketitle

\section{Introduction and main results}\label{sec:intr}

In the Hilbert space $L_2(\bR)$, we consider a Schr\"odinger operator
\[
  S = - \frac{d^2}{dx^2} + q
\]
with potential $q$ that is a real-valued distribution from the space
$W^{-1}_{2,\mathrm{loc}}(\bR)$. Recall that $W^{-1}_{2,\mathrm{loc}}(\bR)$ is the dual space to the space~$W^1_{2,\mathrm{comp}}(\bR)$ of functions in $W_2^1(\bR)$ with compact support and that every real-valued $q\in W^{-1}_{2,\mathrm{loc}}(\bR)$ can be represented as $\si'$ for a real-valued function~$\si$ from $L_{2,\mathrm{loc}}(\bR)$. The operator $S$ can then be rigorously defined e.g.\ by the so-called regularization method that was used in~\cite{AEZ:1988} in the particular case $q(x)=1/x$ and then developed for generic
distributional potentials in~$W_{2,\mathrm{loc}}^{-1}(\bR)$ by Savchuk and Shkalikov~\cite{SavShk:1999, SavShk:2003}; see also recent extensions to more general differential expressions in~\cite{GorMik:2010,GorMik:2011}. Namely, the regularization method suggests to define $S$ via
\begin{equation}\label{eq:S}
  S f = \ell (f):=  - (f' -\si f)' - \si f'
\end{equation}
on the natural maximal domain
\begin{equation}\label{eq:dom-S}
  \dom S = \{f \in L_2(\mathbb R) \mid f, \
    f'-\sigma f \in AC_{\mathrm{loc}}(\mathbb R),\
      \ell(f) \in L_2(\mathbb R)\};
\end{equation}
here $AC_{\mathrm{loc}}(\bR)$ is the space of functions that are locally absolutely continuous.
It is straightforward to see that $Sf = -f'' + qf$ in the sense of distributions, so that the above definition is independent of the particular choice of the primitive $\si\in L_{2,\mathrm{loc}}(\bR)$.

One can also introduce the minimal operator $S_0$, which is the closure of the restriction~$S_0'$ of $S$ onto the set of functions of compact support, i.e., onto
\[
    \dom S_0' = \{f \in L_{2,\mathrm{comp}}(\bR) \mid f, f'-\si f \in AC_{\mathrm{loc}}(\bR),\ \ell(f) \in L_2(\bR)\}.
\]
The operator $S_0'$ (and hence $S_0$) is symmetric; moreover, in a standard manner~\cite{MM:2008} one proves that $S$ is the adjoint of~$S_0$, so that $S$ is the so-called maximal operator.

An important question preceding any further analysis of the operator~$S$ is whether it is self-adjoint. Recently, this question has attracted attention in the literature in the particular case where the distributional potential $q\in W^{-1}_{2,\mathrm{loc}}(\bR)$ contains the sum of Dirac delta-functions~\cite{AlbKosMal:2011,KosMal:2010,IsmKos:2010} or is periodic~\cite{MM:2008} (complex-valued periodic~$q$ are discussed in~\cite{DjaMit:2010}), or belongs to the space $W^{-1}_{2,\mathrm{unif}}(\bR)$~\cite{HM:2001}. We recall~\cite{HM:2001} that any $q \in W^{-1}_{2,\mathrm{unif}}(\bR)$ can be represented (not uniquely) in the form
$q = \si' + \tau$, where $\si$ and $\tau$ belong to $L_{2,\mathrm{unif}}(\bR)$ and  $L_{1,\mathrm{unif}}(\bR)$, respectively, i.e.,
\begin{align*}
  \|\sigma\|^2_{2,\mathrm{unif}} &:= \sup_{t\in\bR} \int_t^{t+1} |\sigma(s)|^2 ds < \infty,\\
   \|\tau\|_{1,\mathrm{unif}} &:= \sup_{t\in\bR} \int_t^{t+1} |\tau(s)|\,ds < \infty,
\end{align*}
and the derivative is understood in the sense of distributions.
Given such a representation, the operator $S$ is defined as
\begin{equation}\label{eq:S-unif}
    Sf= -(f'-\si f)' - \si f' + \tau f
\end{equation}
on the domain~\eqref{eq:dom-S}; this definition is again independent of the particular choice of $\sigma$ and $\tau$ above.

Theorem~3.5 of our paper~\cite{HM:2001} claims that for real-valued $q \in W^{-1}_{2,\mathrm{unif}}$ the operator~$S$ as defined by~\eqref{eq:S-unif} and \eqref{eq:dom-S} is self-adjoint and coincides with the operator~$T$ constructed by the form-sum method. However, as was pointed out in~\cite{MM:2008} and \cite{Gesz:2011}, the proof given in~\cite{HM:2001} is incomplete: namely, it establishes the inclusion $T\subset S$ but then derives the equality $S=T$ taking for granted that $S$ is symmetric.
However, since $S_0$ is symmetric, symmetry of~$S$ would immediately imply its self-adjointness, and only the claim that $S=T$ in Theorem~3.5 of~\cite{HM:2001} would remain non-trivial.

The fact that $S$ is indeed self-adjoint is rigorously justified in the paper~\cite{MM:2008} for the particular case where $q\in W_{2,\mathrm{unif}}^{-1}(\bR)$ is periodic. The authors prove therein that $S_0$, $S$, $T$, and the Friedrichs extension of $S_0$ all coincide; however, the arguments heavily use periodicity of~$q$ and thus are not applicable for generic real-valued $q\in W^{-1}_{2,\mathrm{unif}}(\bR)$.

Recently, Albeverio, Kostenko and Malamud~\cite{AlbKosMal:2011} extended the Povzner--Wienholtz theorem  stating that boundedness below of the minimal operator implies its self-adjointness (see~\cite{ClaGes:2003} and the references therein) to the class of arbitrary distributional potentials in $W^{-1}_{2,\mathrm{loc}}(\bR)$. The proof of Theorem~I.1 in~\cite{AlbKosMal:2011} is for the half-line and for the particular case where $q = q_0 + \sum_k \alpha_k \delta(\cdot-x_k)$, where $q_0 \in L_{1,\mathrm{loc}}(\bR)$, $\alpha_k$ and $x_k$ are real numbers, and $\delta$ is the Dirac delta-function; however, Remark~III.2 explains that the same proof works in the more general situation of $q\in W^{-1}_{2,\mathrm{loc}}(\bR)$. In particular, for $q \in W^{-1}_{2,\mathrm{unif}}(\bR)$ the minimal operator $S_0$ is shown in~\cite{HM:2001} to be bounded below; therefore, the operator $S_0=S$ is then self-adjoint by the above extension of the Povzner--Wienholtz theorem. This fills out the gap in the proof of Theorem~3.5 of our paper~\cite{HM:2001}.

The aim of this note is to give an alternative proof of the Povzner--Wienholtz theorem for distributional potentials $q \in W^{-1}_{2,\mathrm{loc}}(\bR)$. Our approach has several merits; namely, it gives the representation of a positive operator $S$ in the von Neumann form $A^*A$ for some first order differential operator~$A$ and provides additional information on the domain of~$S$. For regular~$q$, possibility of such a representation is known to follow from disconjugacy of $S$ on the whole line, i.e., from the Jacobi condition in the variational problem for the corresponding quadratic form of~$S$, see~\cite[Ch.~XI.10,11]{Har:1982}. We also mention that the factorization of~$S$ as $A^*A$ is of basic importance for the Darboux transformation method, also called Darboux--Crum, or single commutation method, see~\cite{Cru:1955,Dar:1882,Dei:1978,MatSal:1991}.

Namely, assume that a real-valued distribution $q \in W^{-1}_{2,\mathrm{loc}}(\bR)$ is such that the minimal operator $S_0$ is bounded below. Adding a constant to $q$ as necessary, we can make $S_0$ positive and shall assume this throughout the rest of the note. Then~\cite{KPST:2005} the equation
$y'' = qy$ has a (possibly not unique) solution that is positive over $\bR$, and $r:=y'/y \in L_{2,\mathrm{loc}}(\bR)$ is a global distributional solution to the Riccati equation $r'+r^2=q$. The function~$r$ is called the \emph{Riccati representative} of~$q$. Moreover, the differential expression~$\ell$ of~\eqref{eq:S}  admits then a formal representation
\[
    \ell := -\frac{d^2}{dx^2} + q = - \Bigl(\frac{d}{dx}+r\Bigr)\Bigl(\frac{d}{dx}-r\Bigr).
\]
This representation suggests that $\ell$ is also related to a differential operator~$A^*A$, where
$A$ is the differential operator of first order given by
\begin{equation}\label{eq:A-oper}
    Af = f' - rf
\end{equation}
on the maximal domain
\begin{equation}\label{eq:A-dom}
    \dom A = \{f \in L_2(\bR) \mid f'-rf \in L_2(\bR)\}.
\end{equation}
The derivative $f'$ for $f\in\dom A$ is understood in the sense of distributions; observe, however, that $f' = rf + Af$ is locally integrable so that every $f\in \dom A$ is locally absolutely continuous.

Our extension of the Povzner--Wienholtz theorem reads now as follows.

\begin{theorem}\label{thm:main}
Assume that a real-valued distribution $q \in W^{-1}_{2,\mathrm{loc}}(\bR)$ is such that the minimal operator $S_0$ is positive and denote by $r \in L_{2,\mathrm{loc}}(\bR)$ a Riccati representative of~$q$. Then $S_0$ is self-adjoint; moreover, $S_0=S=A^*A$, and for every $f\in \dom S$ it holds that $f'-rf \in L_2(\bR)$.
\end{theorem}

This theorem can further be specified if $q \in W^{-1}_{2,\mathrm{unif}}(\bR)$. As we mentioned above, the operator $S_0$ is then automatically bounded below and thus self-adjoint; moreover, we can characterize its domain as follows.

\begin{corollary}\label{cor:dom-S}
Assume that a real-valued $q \in W^{-1}_{2,\mathrm{unif}}(\bR)$ is written as $q = \sigma' + \tau$ with some $\sigma \in L_{2,\mathrm{unif}}(\bR)$ and $\tau \in L_{1,\mathrm{unif}}(\bR)$. Then the corresponding maximal Schr\"odinger operator~$S$ is self-adjoint; moreover, $\dom S \subset W_2^1(\bR)$ and  $y'-\sigma y \in L_2(\bR)$ for every $y\in \dom S$.
\end{corollary}

We observe that Proposition~12 of \cite{MM:2008} shows that if $q\in W^{-1}_{2,\mathrm{loc}}(\bR)$ is periodic, then the three statements:
\begin{itemize}
 \item[(a)] $S$ is self-adjoint;
 \item[(b)] $\dom S \subset W_2^1(\bR)$;
 \item[(c)] for every $y\in\dom S$, $y'-\sigma y \in L_2(\bR)\cap AC_{\mathrm{loc}}(\bR)$
\end{itemize}
are equivalent.

\section{Proofs}

We start with the following simple observation.

\begin{lemma}
The operator~$A$ defined in~\eqref{eq:A-oper}--\eqref{eq:A-dom} is closed.
\end{lemma}

\begin{proof}
Let $y_n\in\dom A$ be such that $y_n \to y$ and $g_n:=Ay_n \to g$ in $L_2(\bR)$ as $n\to\infty$.
Since convergence in $L_{1,\mathrm{loc}}(\bR)$ yields convergence in the space of distributions $\sD'(\bR)$, we conclude that $y_n \to y$, $ry_n \to ry$, and $g_n \to g$ in $\sD'(\bR)$. Therefore,
$y_n' = ry_n + g_n \to ry + g$ in $\sD'(\bR)$ as $n\to\infty$; on the other hand, $y_n' \to y'$ in $\sD'(\bR)$ since differentiation is a continuous operation in $\sD'(\bR)$. It follows that $y' = ry+g$, whence $y\in \dom A$ and $Ay=g$ as required.
\end{proof}

The von Neumann theorem~\cite[Thm.~V.3.24]{Kat:1976} yields now the following result.

\begin{corollary}\label{cor:SF}
The operator $S_F:= A^* A$ is self-adjoint on the domain
\[
    \dom S_F := \{ f\in L_2(\bR) \mid Af \in \dom A^*\}.
\]
\end{corollary}

Clearly, $S_F$ is a self-adjoint extension of the minimal operator $S_0$. It turns out that $S_F$ is the Friedrichs extension of $S_0$, see Chapter~VI of Kato's classic book~\cite{Kat:1976} for all relevant definitions.

\begin{lemma}\label{lem:Fri}
The operator~$S_F$ is the Friedrichs extension of $S_0$.
\end{lemma}

\begin{proof}
We recall that the Friedrichs extension of $S_0$ is the self-adjoint operator associated with the closure $\mathfrak{s}_0$ of the quadratic form of $S_0$ (defined initially on $\dom S_0$) via the first representation theorem~\cite[Thm.~VI.2.1]{Kat:1976}. The quadratic form $\mathfrak{s}_F$ of $S_F$ is an extension of~$\mathfrak{s}_0$, and to prove that $\mathfrak{s}_0=\mathfrak{s}_F$ it suffices to show that $\dom S_0$ is a core for~$\mathfrak{s}_F$.

It is straightforward to see that $\dom\mathfrak{s}_F$ coincides with $\dom A$ and that $\mathfrak{s}_F$-convergence is equivalent to the $A$-convergence. Therefore it suffices to show that $\dom S_0$ is a core for~$A$. By the von Neumann theorem~\cite[Thm.~V.3.24]{Kat:1976} $\dom A^*A$ is a core for $A$, and it suffices to show that $\dom S_0$ is dense in $\dom A^*A$ in the graph topology of~$A$.

To this end let $f\in \dom A^*A$ be arbitrary. Take $\chi \in C_0^\infty$ such that $0 \le \chi \le 1$ and $\chi\equiv1$ on $(-1,1)$, and set $\chi_n:=\chi(\cdot/n)$ and $f_n:=\chi_nf$. Then $f_n\to f$ and $Af_n=\chi_n(Af) + f\chi_n' \to Af$ in $L_2(\bR)$ as $n\to\infty$, i.e., $f_n$ converge to $f$ in the graph topology of~$A$. Since $Af\in \dom A^*$, we see that $Af_n = f_n'-rf_n$ is absolutely continuous. Recalling that $r'+r^2 = \si'$, we conclude that $r-\si$ is locally absolutely continuous, whence $f_n'-\si f_n$ is absolutely continuous as well. Thus $f_n$ belong to the domain of~$S_0'$, which is henceforth dense in $\dom A^*A$ in the graph topology of~$A$, and the proof is complete.
\end{proof}

Now we study the maximal operator~$S$. The first observation is as follows.

\begin{lemma}\label{lem:dom-S}
For every $y\in \dom S$, the quasi-derivative $y^{[1]}:=y'-ry$ belongs to $L_2(\bR)$.
\end{lemma}

\begin{proof}
Set $g:= Sy$ and assume that $y^{[1]}=y'-ry$ is not in $L_2(\bR^+)$. Integrating $\ell (y)\ov{y}=g\ov{y}$ by parts from $0$ to $x$, we find that
\[
    \int_0^x g(t)\overline{y}(t)\,dt = \int_0^x |y^{[1]}(t)|^2\,dt - y^{[1]}(x)\ov{y}(x) + y^{[1]}(0)\ov{y}(0).
\]
It follows that
\[
    \frac1T \int_0^T \int_0^x |y^{[1]}(t)|^2\,dt\,dx - \frac1T\int_0^T  y^{[1]}(x)\ov{y}(x)\,dx
        = \frac1T \int_0^T \int_0^x g(t)\overline{y}(t)\,dt\, dt - y^{[1]}(0)\ov{y}(0)
\]
remains bounded as $T\to\infty$; since $\int_0^x |y^{[1]}(t)|^2\,dt$ grows to $+\infty$ as $x\to\infty$ by assumption, we conclude that
\[
    \frac1T\Bigl|\int_0^T y^{[1]}(x)\ov{y}(x)\,dx\Bigr| \to \infty
\]
as $T\to\infty$ and, moreover, that
\begin{equation}\label{eq:int-ineq}
    2 \Bigl|\int_0^T  y^{[1]}(x)\ov{y}(x)\,dx\Bigr| \ge \int_0^T \int_0^x |y^{[1]}(t)|^2\,dt\,dx
\end{equation}
for all $T$ large enough. In view of the Cauchy--Bunyakovsky--Schwarz inequality
\[
     \Bigl|\int_0^T  y^{[1]}(x)\ov{y}(x)\,dx\Bigr|
        \le \|y\| \Bigl(\,\int_0^T  |y^{[1]}(x)|^2\,dx\Bigr)^{1/2},
\]
\eqref{eq:int-ineq} results in the inequality
\[
    \int_0^T  |y^{[1]}(x)|^2\,dx
        \ge \frac1{4\|y\|^2}\Bigl(\int_0^T\int_0^x  |y^{[1]}(t)|^2\,dt\,dx\Bigr)^{2}.
\]
Set $I(T):= \int_0^T\int_0^x  |y^{[1]}(t)|^2\,dt\,dx$; then the above inequality can be written as
\[
    I'(T) \ge \frac1{4\|y\|^2}I^2(T),
\]
and, upon integration, yields
\begin{equation}\label{eq:int-ineq2}
    \frac1{I(T_0)} - \frac1{I(T)} \ge \frac{T-T_0}{4\|y\|^2}
\end{equation}
for every positive $T$ and $T_0$ such that $T>T_0$ and $I(T_0)>0$.
However, the assumption that $y^{[1]}\not\in L_2(\bR^+)$ implies that $I(T)\to\infty$ as $T\to\infty$, which is in contradiction with~\eqref{eq:int-ineq2}. Therefore $y^{[1]} \in L_2(\bR^+)$; the fact that $y^{[1]} \in L_2(\bR^-)$ is proved analogously.
\end{proof}

\begin{remark}
Similar arguments were used in~\cite[Lemma~XI.7.1]{Har:1982} and \cite[Lemma~4.1]{KPST:2005} in the study of the Riccati equation.
\end{remark}

\begin{proof}[Proof of Theorem~\ref{thm:main}]
By Lemma~\ref{lem:dom-S}, $\dom S \subset \dom A$. Further, $\dom A = \dom \mathfrak{s}_F$, where $\mathfrak{s}_F$ is the quadratic form of $S_F$, the Friedrichs extension of $S_0$. By the extremal property of the Friedrichs extension~\cite[Thm.~VI.2.11]{Kat:1976} we conclude that every self-adjoint restriction of $S$, i.e., every self-adjoint extension of~$S_0$, coincides with $S_F$. This implies that the minimal operator $S_0$ is itself self-adjoint and that $S_0=S_F =S$ as claimed.
\end{proof}

It was proved in~\cite{HM:2001} that if $q\in W^{-1}_{2,\mathrm{unif}}(\bR)$, then the operator $S_0$ is bounded below. Assuming that $S_0$ is already positive, we have as before $q = r' + r^2$ for some $r\in L_{2,\mathrm{loc}}(\bR)$. It turns out that the function $r$ in this representation has some special properties.

\begin{lemma}\label{lem:Riccati}
Assume that real-valued $q \in W^{-1}_{2,\mathrm{unif}}(\bR)$ and $r \in L_{2,\mathrm{loc}}(\bR)$ satisfy the equation $r'+r^2 = q$ in the sense of distributions. Then $r\in L_{2,\mathrm{unif}}(\bR)$.
\end{lemma}

\begin{proof}
We set
\[
    a_n:=\int_{n}^{n+1} r^2(t)\,dt, \qquad n\in\bZ,
\]
and prove that $\sup_{n\in\bZ}a_n$ is finite.

Denote by $\phi$ the function in $W^1_2(\bR)$ with support equal to $[-1,2]$ and defined via
\[
    \phi(x) = \begin{cases}
        1+x & x \in [-1,0),\\
        1   & x \in [0,1],\\
        2-x & x \in (1,2].
    \end{cases}
\]
We also set $\phi_\xi:=\phi(\,\cdot\,-\xi)$ and notice that $\|\phi_\xi\|_{L_\infty}=\|\phi'_\xi\|_{L_\infty}=1$.
Denoting by $\langle \,\cdot\,,\,\cdot\,\rangle$ the pairing  between $W^{-1}_{2,\mathrm{loc}}(\bR)$ and $W^1_{2,\mathrm{comp}}(\bR)$, we find that
\begin{equation}\label{eq:Riccati}
    -\langle r,\phi'_\xi\rangle + \langle r^2,\phi_\xi\rangle = \langle q,\phi_\xi\rangle.
\end{equation}
As $q = \sigma'+\tau$ with some $\sigma\in L_{2,\mathrm{unif}}(\bR)$ and $\tau\in L_{1,\mathrm{unif}}(\bR)$, the right-hand side of this equality admits the uniform estimate
\begin{equation}\label{eq:q-unif}
    |\langle q,\phi_\xi\rangle| \le |\langle \sigma,\phi'_\xi\rangle|
                        + |\langle \tau,\phi_\xi\rangle|
                        \le 3 \|\sigma\|_{2,\mathrm{unif}} + 3 \|\tau\|_{1,\mathrm{unif}} =: C;
\end{equation}
we assume that $C>0$ as otherwise $q\equiv r\equiv 0$ and there is nothing to prove.
The inequalities
\[
    \langle r^2 ,\phi_n\rangle  \ge a_n, \qquad
    |\langle r ,\phi'_n\rangle| \le a^{1/2}_{n-1} + a^{1/2}_{n+1}
\]
combined with~\eqref{eq:Riccati} and \eqref{eq:q-unif} lead to the relation
\begin{equation}\label{eq:an-bound}
    a_n \le a^{1/2}_{n-1} + a^{1/2}_{n+1} +C.
\end{equation}

We shall prove below that
\begin{equation}\label{eq:liminf}
    \liminf_{n\to-\infty}a_n \le C/2, \qquad \liminf_{n\to+\infty}a_n \le C/2,
\end{equation}
so that there exist sequences $(n^-_k)_{k\in\bN}$ and $(n^+_k)_{k\in\bN}$ tending respectively to $-\infty$ and $+\infty$ such that $a_{n^\pm_k}<C$ for all $k\in\bN$.
Given this, the proof is concluded as follows. We have either $a_n\le C$ for all $n\in\bZ$, or otherwise $a_m > C$ for some $m\in\bZ$. In the latter case, for every $k$ so large that $m \in (n^-_k,n^+_k)$ the maximum
\[
    C_k:=\max\{a_j \mid j = n^-_k, \dots, n^+_k\}
\]
is assumed for some index $m_k$ strictly between $n^-_k$ and $n^+_k$. Inequality~\eqref{eq:an-bound} for $n=m_k$ then yields
\[
    C_k \le 2C^{1/2}_k + C,
\]
whence $C_k \le 2 C+4$. Therefore in both cases $\sup_{n\in\bZ}a_n$ is finite thus implying that $r\in L_{2,\mathrm{unif}}(\bR)$ as claimed.

It remains to establish~\eqref{eq:liminf}. To this end we take $a<b$ so that $b-a>3$ and integrate~\eqref{eq:Riccati} in $\xi$ over $(a,b)$. As
\[
    \int_a^b \phi_\xi'(t)\,d\xi = \int_a^b \phi'(t-\xi)\,d\xi
        = \phi_a(t)-\phi_b(t),
\]
the Fubini theorem yields
\begin{equation}\label{eq:Fubini}
    -\int_a^b\langle r,\phi'_\xi\rangle\,d\xi = \langle r,\phi_b\rangle
                        -\langle r,\phi_a\rangle.
\end{equation}
Similarly,
\[
    \int_a^b\langle r^2,\phi_\xi\rangle\,d\xi = \langle r^2,\psi\rangle
\]
with
\[
    \psi(t):=\int_a^b \phi_\xi(t) \,d\xi.
\]
Observing that $\operatorname{supp} \psi = [a-1,b+2]$, that $\psi(t) = 2$ for $t\in[a+2,b-1]$ and that
$\psi(t) \ge \tfrac12\phi^2_a(t)$ for $t\in [a-1,a+2]$ and
$\psi(t) \ge \tfrac12\phi^2_b(t)$ for $t\in [b-1,b+2]$,
we get
\[
    \langle r^2,\psi\rangle \ge 2 \int_{a+2}^{b-1} r^2(t)\,dt +
        \tfrac12\langle r^2,\phi^2_a\rangle + \tfrac12\langle r^2,\phi^2_b\rangle.
\]
On the other hand, relations \eqref{eq:Riccati}, \eqref{eq:q-unif}, and \eqref{eq:Fubini}  imply the inequality
\[
    \langle r^2,\psi\rangle \le \Bigl|\int_a^b\langle q,\phi_\xi\rangle\,d\xi\Bigr|
            + \Bigl|\int_a^b \langle r,\phi'_\xi\rangle\,d\xi \Bigr|
                \le C(b-a) + |\langle r,\phi_a\rangle| + |\langle r,\phi_b\rangle|.
\]
Noticing that $|\langle r,\phi_\xi\rangle|\le 2 \langle r^2,\phi^2_\xi\rangle^{1/2}$ by the Cauchy--Bunyakovsky--Schwarz inequality and that $2x-\tfrac12x^2\le 2$ for $x\in\bR$, we conclude that
\begin{align*}
    2 \int_{a+2}^{b-1} r^2(t)\,dt
    &\le C(b-a)
        + 2 \langle r^2,\phi^2_a\rangle^{1/2}
        - \tfrac12\langle r^2,\phi^2_a\rangle
        + 2 \langle r^2,\phi^2_b\rangle^{1/2}
        - \tfrac12\langle r^2,\phi^2_b\rangle\\
    & \le C(b-a)+4.
\end{align*}
This estimate yields~\eqref{eq:liminf} in a straightforward manner, and the proof is complete.
\end{proof}

\begin{proof}[Proof of Corollary~\ref{cor:dom-S}]
We may again assume that the operator $S$ is positive and denote by $r\in L_{2,\mathrm{unif}}(\bR)$ the corresponding solution of the Riccati equation $r'+r^2 =q$ and by $A$ the differential operator of~\eqref{eq:A-oper}--\eqref{eq:A-dom}. By Lemma~\ref{lem:dom-S}, the domain of~$S$ is contained in~$\dom A$, so that it suffices to show that $\dom A \subset W_2^1(\bR)$.

Take an arbitrary $y \in \dom A$;  thus $y$ and $y'-ry=g$ are in $L_2(\bR)$. Set $\Delta_n:=[n,n+1)$, $g_n:=\bigl(\int_{\Delta_n}|g(t)|^2\,dt\bigr)^{1/2}$, and choose $\xi_n \in \Delta_n$ such that
\[
    |y(\xi_n)| \le \Bigl(\int_{\Delta_n}|y(t)|^2\,dt\Bigr)^{1/2} =: y_n.
\]
For every $x\in\Delta_n$, we integrate the equality $y'=ry+g$ from $\xi_n$ to $x$ to get the estimates
\[
    |y(x)| \le |y(\xi_n)| + \int_{\Delta_n} |r(t)y(t)|\,dt + \int_{\Delta_n}|g(t)|\,dt
            \le y_n +  y_n\,\|r\|_{2,\mathrm{unif}} + g_n =: b_n
\]
and
\[
    \int_{\Delta_n} |r(t)y(t)|^2\,dt \le b_n^2\,\|r\|^2_{2,\mathrm{unif}}.
\]
Since the sequence $(b_n)$ belongs to $\ell_2(\bZ)$, it follows that $ry \in L_2(\bR)$; thus $y' = ry + g \in L_2(\bR)$, and $y \in W_2^1(\bR)$.

Further, it was proved in~\cite{HM:2001} that $y\in W_2^1(\bR)$ and $\sigma \in L_{2,\mathrm{unif}}(\bR)$
imply that $\sigma y \in L_2(\bR)$, whence the quasi-derivative $y'-\sigma y$ belongs to $L_2(\bR)$ as well. The proof is complete.
\end{proof}

\medskip

\emph{Acknowledgements.} The authors thank Professors F.~Gesztesy, A.~Kostenko, M.~Malamud, and V.~Mikhailets for fruitful discussions and comments. R.H. acknowledges support from the Isaac Newton Institute for Mathematical Sciences at the University of Cambridge for participation in the programme \emph{``Inverse Problems''}, during which part of this work was done.

\end{document}